\documentclass[reqno]{amsart}
\usepackage{articletemplate}
\usepackage{graphicx}
\usepackage[section]{algorithm}
\usepackage{algorithmic}

\usepackage{comment}

\newcommand{\BlackBoxes}{\global\overfullrule5pt}

\BlackBoxes

\newcommand{\weakD}{\stackrel{\lower0.2ex\hbox{$\scriptscriptstyle
             \mathbf{D}[0,1] $}}{\Rightarrow}}

\makeatletter
\AtBeginDocument{%
\def\@serieslogo{%
\vbox to\headheight{%
\parindent\z@ \fontsize{6}{7\p@}\selectfont
November 9, 2012\endgraf\
\vss}}}
\makeatother

\def\P{\mathbb{P}}
\def\R{\mathbb{R}}

\def\Prob{\mathbf{P}}
\def\E{\mathbf{E}}

\newcommand{\oF}{\overline{F}}
\newcommand{\bx}{\textbf{x}}

\theoremstyle{theorem}
\newtheorem{theorem}{Theorem}[section]

\newtheorem{proposition}[theorem]{Proposition}

\theoremstyle{definition}

\newtheorem{remark}[theorem]{Remark}
\newtheorem{alg}[theorem]{Algorithm}

\newcommand{\pmc}{\widehat{p}}
\newcommand{\pis}{\widehat{p}}

\newcommand{\by}{\textbf{y}}

\newcommand{\bone}{\mathbf{1}}

\begin{document}

\title[MCMC for heavy-tailed random walk]{Markov chain Monte Carlo for computing rare-event probabilities for a heavy-tailed random walk}


\author[T.~Gudmundsson]{Thorbj\"orn Gudmundsson}
\address{Department of Mathematics, KTH, 100 44
Stockholm, Sweden}
\email{{tgud@kth.se, hult@kth.se}}
\author[H.~Hult]{Henrik Hult$^{\dagger}$}

\copyrightinfo{}

\thanks{$^{\dagger}$ Henrik Hult's research was supported by the G\"oran Gustafsson Foundation}

\begin{abstract}
In this paper a method based on a Markov chain Monte Carlo (MCMC) algorithm is proposed to compute the probability of a rare event. The conditional distribution of the underlying process given that the rare event occurs has the probability of the rare event as its normalizing constant. Using the MCMC methodology a Markov chain is simulated,
with that conditional distribution as its invariant distribution, and information about the normalizing constant is extracted from its trajectory. The algorithm is described in full generality and applied to the problem of computing the probability that a heavy-tailed random walk exceeds a high threshold. An unbiased estimator of the reciprocal probability is constructed whose normalized variance vanishes asymptotically. The algorithm is extended to random sums and its performance is illustrated numerically and compared to existing importance sampling algorithms.
\end{abstract}


\maketitle


\noindent {\small
}

\section{Introduction}
\label{sec:intro}

In this paper a Markov chain Monte Carlo (MCMC) methodology is proposed for computing the probability of a rare event. The basic idea is to use an MCMC algorithm to sample from the conditional distribution given the event of interest and then extract the probability of the event as the normalizing constant. The methodology will be outlined in full generality and exemplified in the setting of computing hitting probabilities for a heavy-tailed random walk.

A rare-event simulation problem can be often be formulated as follows. Consider a sequence of random variables $X^{(1)}, X^{(2)}, \dots,$ each of which can be sampled repeatedly by a simulation algorithm.  The objective is to estimate $p^{(n)} = \Prob(X^{(n)} \in A)$, for some large $n$, based on a sample $X^{(n)}_{0}, \dots, X^{(n)}_{T-1}$. It is assumed that the probability $\Prob(X^{(n)} \in A) \to 0$, as $n\to \infty$, so that the event $\{X^{(n)} \in A\}$ can be thought of as rare.  The solution to the problem consists of finding a family of simulation algorithms and corresponding estimators  whose performance is satisfactory for all $n$. For unbiased estimators $\widehat p^{(n)}_{T}$ of $p^{(n)}$ a useful performance measure is the relative error:
\begin{align*}
	\text{RE}^{(n)} = \frac{\Var(\widehat{p}_{T}^{(n)})}{(p^{(n)})^{2}}.
\end{align*}
An algorithm is said to have \emph{vanishing relative error} if the relative error tends to zero as $n \rightarrow \infty$ and \emph{bounded relative error} if the relative error is bounded in $n$.

It is well known that the standard Monte Carlo algorithm is inefficient for computing rare-event probabilities. As an illustration, consider the standard Monte Carlo estimate
\begin{equation*}
\pmc_{T}^{(n)} = \frac{1}{T}\sum_{t=0}^{T-1} I\{ X_t^{(n)} \in A \},
\end{equation*}
of $p^{(n)} = \Prob(X^{(n)} \in A)$ based on independent replicates $X_{0}^{(n)}, \dots, X_{T-1}^{(n)}$. The relative error of the Monte Carlo estimator is
$$ \frac{\Var(\pmc^{(n)}_{T})}{(p^{(n)})^2} = \frac{p^{(n)}(1-p^{(n)})}{T(p^{(n)})^2} = \frac{1}{Tp^{(n)}} - \frac{1}{T} \to \infty,$$
as $n \to \infty$, indicating that the performance deteriorates when the event is rare.

A popular method to reduce the computational cost is importance sampling, see e.g.\ \cite{Asmussen2007}. In importance sampling the random variables  $X_{0}^{(n)}, \dots, X_{T-1}^{(n)}$ are sampled independently from a different distribution, say $G^{(n)}$, instead of the original distribution $F^{(n)}$. The importance sampling estimator is defined as a weighted empirical estimator,
\begin{equation*}
\pis^{(n)}_{T} = \frac{1}{T} \sum_{t=0}^{T-1} L^{(n)}(X^{(n)}_{t}) I\{ X^{(n)}_{t} \in A \},
\end{equation*}
where $L^{(n)} = dF^{(n)}/dG^{(n)}$ is the likelihood ratio, which is assumed to exist on $A$. The importance sampling estimator $\pis_{T}^{(n)}$ is unbiased and its performance depends on the choice of the sampling distribution $G^{(n)}$.
The optimal sampling distribution is called the zero-variance distribution and is simply
the conditional distribution,
\begin{equation*}
F^{(n)}_A(\cdot) = \Prob(X^{(n)} \in \cdot \mid  X^{(n)} \in A) = \frac{\Prob( X^{(n)} \in \cdot  \, \cap A )}{p^{(n)}} \text{.}
\end{equation*}
In this case the likelihood ratio weights $L^{(n)}$ are equal to $p^{(n)}$ which implies that $\pis^{(n)}_{T}$ has zero variance. Clearly, the zero-variance distribution cannot be implemented in practice, because $p^{(n)}$ is unknown, but it serves as a starting point for selecting the sampling distribution.  A good  idea is to choose a sampling distribution $G^{(n)}$ that approximates the zero-variance distribution and such that the random variable $X^{(n)}$ can easily be sampled from $G^{(n)}$, the event $\{X^{(n)} \in A \}$ is more likely under the sampling distribution $G^{(n)}$ than under the original $F^{(n)}$, and the likelihood ratio $L^{(n)}$ is unlikely to become too large. Proving efficiency (e.g.\ bounded relative error) of an importance sampling algorithm can be technically cumbersome and often requires extensive analysis.

The methodology proposed in this paper is also based on the conditional distribution $F^{(n)}_{A}$. Because $F^{(n)}_{A}$ is known up to the normalizing constant $p^{(n)}$ it is possible to sample from $F^{(n)}_{A}$ using an MCMC algorithm such as a Gibbs sampler or Metropolis-Hastings algorithm. The idea is to generate samples $X_{0}^{(n)}, \dots, X_{T-1}^{(n)}$ from a Markov chain with stationary distribution $F^{(n)}_{A}$ and construct an estimator of the normalizing constant $p^{(n)}$. An unbiased estimator of $(p^{(n)})^{-1}$ is constructed from a known probability density $v^{(n)}$ on $A$, which is part of the design, and the original density $f^{(n)}$ of $X^{(n)}$ by
\begin{align}\label{intro:phat}
	\widehat{q}^{(n)}_{T} = \frac{1}{T} \sum_{t=0}^{T-1} \frac{v^{(n)}(X_{t}^{(n)})I\{X_{t}^{(n)} \in A\}}{f^{(n)}(X_{t}^{(n)})}.
\end{align}
The performance of the estimator depends both on the choice of the density $v^{(n)}$ and on the ergodic properties of the MCMC sampler used in the implementation. Roughly speaking the rare-event properties, as $n \to \infty$, are controlled by the choice of $v^{(n)}$ and the large sample properties, as $T \to \infty$, are controlled by the ergodic properties of the MCMC sampler.

The computation of normalizing constants and ratios of normalizing constants in the context of MCMC is a reasonably well studied problem in the statistical literature, see e.g.\  \cite{GM98} and the references therein. However, such methods have, to the best of our knowledge, not been studied in the context of rare-event simulation.

To exemplify the MCMC methodology we consider the problem of computing the probability that a random walk $S_{n} = Y_{1} + \dots + Y_{n}$, where $Y_{1}, \dots, Y_{n}$ are nonnegative, independent, and heavy-tailed random variables,  exceeds a high threshold $a_{n}$. This problem has received some attention in the context of conditional Monte Carlo algorithms \cite{Asmussen1997,Asmussen2006} and importance sampling algorithms \cite{Juneja2002,Dupuis2007,BlanchetLiu2008,BL10}.

In this paper a Gibbs sampler is presented for sampling from the conditional distribution $\Prob((Y_{1}, \dots, Y_{n}) \in \cdot \mid S_{n} > a_{n})$. The resulting Markov chain is proved to be uniformly ergodic. An estimator for $(p^{(n)})^{-1}$ of the form \eqref{intro:phat} is suggested with $v^{(n)}$ as the conditional density of $(Y_{1}, \dots, Y_{n})$ given $\max\{Y_{1}, \dots, Y_{n}\} > a_{n}$. The estimator is proved to have vanishing normalized variance when the distribution of $Y_{1}$ belongs to the class of subexponential distributions. The proof is elementary and is completed in a few lines. This is in sharp contrast to efficiency proofs for importance sampling algorithms for the same problem, which require more restrictive assumptions on the tail of $Y_{1}$ and tend to be long and technical \cite{Dupuis2007,BlanchetLiu2008,BL10}. An extension of the algorithm to a sum with a random number of steps is also presented.

Here follows an outline of the paper.
The basic methodology and a heuristic efficiency analysis for computing rare-event probabilities is described in Section \ref{sec:mcmc:prob}. The general formulation for computing expectations is given in Section \ref{sec:mcmc:exp} along with a precise formulation of the efficiency criteria. Section \ref{sec:htrw} contains the design and efficiency results for the estimator for computing hitting probabilities for a heavy-tailed random walk, with deterministic and random number of steps. Section \ref{sec:numeric} presents numerical experiments and compares the efficiency of the MCMC estimator against an existing importance sampling algorithm and standard Monte Carlo. The  MCMC estimator has strikingly better performance than existing importance sampling algorithms.

\section{Computing rare-event probabilities by Markov chain Monte Carlo}
\label{sec:mcmc:prob}

In this section an algorithm for computing rare-event probabilities using Markov chain Monte Carlo (MCMC) is presented and conditions that ensure good convergence are discussed in a heuristic fashion. A more general version of the algorithm, for computing expectations, is provided in Section \ref{sec:mcmc:exp} along with a precise asymptotic efficiency criteria.

\subsection{Formulation of the algorithm}

Let $X$ be a real-valued random variable with distribution $F$ and density $f$ with respect to the Lebesgue measure. The problem is to compute the probability
\begin{equation} \label{obj:prob}
p = \Prob( X \in A) = \int_A dF \text{.}
\end{equation}
The event $\{X \in A\}$ is thought of as rare in the sense that $p$ is  small. Let $F_{A}$ be the conditional distribution of $X$ given $X \in A$. The density of $F_{A}$ is given by
\begin{equation} \label{eq:prob:density}
	\frac{dF_{A}}{dx}(x) = \frac{f(x) I\{ x \in A \}}{p}.
\end{equation}
Consider a Markov chain $(X_t)_{t\geq 0}$ whose invariant density is given by \eqref{eq:prob:density}.  Such a Markov chain can be constructed by implementing an MCMC algorithm such as a Gibbs sampler or a Metropolis-Hastings algorithm, see e.g. \cite{Asmussen2007,Gilks1996}.

To construct an estimator for the normalizing constant $p$, consider a non-negative function $v$, which is normalized in the sense that $\int_{A} v(x) dx =1$.  The function $v$ will be chosen later as part of the design of the estimator. For any choice of $v$ the sample mean,  $$\frac{1}{T} \sum_{t=0}^{T-1} \frac{v(X_{t})I\{X_{t} \in A\}}{f(X_{t})},$$ can be viewed as an estimate of
$$ \E_{F_A}\left[ \frac{v(X)I\{X \in A\}}{f(X)} \right] = \int_{A} \frac{v(x)}{f(x)} \frac{f(x)}{p} dx  = \frac{1}{p} \int_{A} v(x)dx = \frac{1}{p}. $$
Thus,
\begin{equation}
\widehat{q}_{T}= \frac{1}{T} \sum_{t=0}^{T-1} u(X_{t}),  \quad \text{where} \quad u(X_{t}) = \frac{v(X_{t})I\{X_{t} \in A\}}{f(X_{t})}, \label{eq:hatpn}
\end{equation}
is an unbiased estimator of $q = p^{-1}$. Then $\widehat{p}_{T} = \widehat{q}_{T}^{-1}$ is an estimator of $p$.

The expected value above is computed under the invariant distribution $F_{A}$ of the Markov chain. It is implicitly assumed that the sample size $T$ is  sufficiently large that the burn-in period, the time until the Markov chain reaches stationarity, is negligible or alternatively that the burn-in period is discarded. Another remark is that it is theoretically possible that all the terms in the sum in \eqref{eq:hatpn} are zero, leading to the estimate $\widehat{q}_{T} = 0$ and then $\widehat p_{T} = \infty$. To avoid such nonsense one can simply take $\widehat{p}_{T}$ as the minimum of  $\widehat{q}_{T}^{-1}$ and one.

There are two essential design choices that determine the performance of the algorithm: the  choice of the function $v$ and the design of the MCMC sampler. The function $v$ influences the variance of $u(X_{t})$ in \eqref{eq:hatpn} and is therefore of main concern for controlling the rare-event properties of the algorithm. It is desirable to take $v$ such that the normalized variance of the estimator, given by $p^{2}\Var(\widehat{q}_{T})$, is not too large.
The design of the MCMC sampler, on the other hand, is  crucial to control the dependence of the Markov chain and thereby the convergence rate of the algorithm as a function of the sample size. To speed up simulation it is desirable that the Markov chain mixes fast so that the dependence dies out quickly.

\subsection{Controlling the normalized variance}

This section contains a discussion on how to control the performance of the estimator $\widehat q_{T}$ by controlling its normalized variance.

For the estimator $\widehat q_{T}$ to be useful it is of course important that its variance is not too large. When the probability $p$ to be estimated is small it is reasonable to ask that $\Var(\widehat{q}_{T})$ is of size comparable to $q^{2} = p^{-2}$, or equivalently, that the standard deviation of the estimator is roughly of the same size as $p^{-1}$. To this end the normalized variance $p^{2} \Var(\widehat q_{T})$ is studied. 

Let us consider $\Var(\widehat q_{T})$.
With
\begin{align*}
u(x) = \frac{v(x)I\{x \in A\}}{f(x)},
\end{align*}
it follows that
\begin{align}
p^{2}\Var_{F_A} ( \widehat q_{T} ) & = p^{2} \Var_{F_A} \Big(\frac{1}{T} \sum_{t=0}^{T-1} u(X_t) \Big)  \nonumber \\
& = p^2 \Big( \frac{1}{T}\Var_{F_A}( u(X_{0})) + \frac{2}{T^2}\sum_{t=0}^{T-1} \sum_{s=t+1}^{T-1} \Cov_{F_A}(u(X_s),u(X_t)) \Big).\label{eq:prob:variance}
\end{align}
Let us for the moment focus our attention on the first term. It can be written as
\begin{eqnarray*}
\frac{p^2}{T} \Var_{F_A} \big( u(X_{0})\big) & = & \frac{p^2}{T}\Big( \E_{F_A} \big[ u(X_{0})^2 \big] - \E_{F_A} \big[u(X_{0})\big]^2 \Big) \\
& = & \frac{p^2}{T} \Big( \int \Big( \frac{v(x)}{f(x)} I\{x \in A\}\Big)^2 F_A(dx) - \frac{1}{p^2}\Big) \\
& = & \frac{p^2}{T} \Big( \int \frac{v^2(x)}{f^2(x)} I\{x \in A\} \frac{f(x)}{p} dx - \frac{1}{p^2}\Big) \\
& = & \frac{1}{T} \Big( \int_A \frac{v^2(x)p}{f(x)} dx - 1 \Big) \text{.}
\end{eqnarray*}

Therefore, in order to control the normalized variance  the function $v$ must be chosen so that
$\int_A \frac{v^2(x)}{f(x)} dx$ is close to $p^{-1}$. An important observation is that the conditional density \eqref{eq:prob:density} plays a key role in finding a good choice of $v$.
Letting $v$ be the  conditional density in \eqref{eq:prob:density} leads to
$$ \int_A \frac{v^2(x)}{f(x)} dx = \int_A \frac{f^2(x)I\{x\in A\} }{p^2f(x)} dx = \frac{1}{p^{2}} \int_A f(x) dx = \frac{1}{p} \text{,} $$
which implies,
$$ \frac{p^2}{T} \Var_{F_A} \big( u(X)\big) = 0 \text{.} $$
This motivates taking $v$ as an approximation of the conditional density \eqref{eq:prob:density}.

If for some set $B \subset A$ the probability $\Prob(X \in B)$ can be computed explicitly, then a candidate for $v$ is
\begin{align*}
	v(x) = \frac{f(x) I\{x \in B\}}{\Prob(X \in B)};
\end{align*}
the conditional density of $X$ given $X \in B$. This candidate is likely to perform well if $\Prob(X\in B)$ is good approximation of $p$. Indeed, in this case
$$ \int_A \frac{v^2(x)}{f(x)} dx = \int_A \frac{f^2(x)I\{x\in B\} }{\Prob(X\in B)^{2} f(x) } dx = \frac{1}{\Prob(X\in B)^{2}} \int_B f(x) dx = \frac{1}{\Prob(X\in B)} \text{,} $$
which will be close to $p^{-1}$.

Now, let us shift emphasis to the covariance term in \eqref{eq:prob:variance}.
As the samples $(X_{t})_{t=0}^{T-1}$ form a Markov chain the $X_{t}$'s are dependent. Therefore the covariance term in \eqref{eq:prob:variance} is non-zero and may not be ignored. The crude upper bound
\begin{align*}
	\Cov_{F_A}(u(X_s),u(X_t)) \leq \Var_{F_{A}}(u(X_{0})),
\end{align*}
leads to the upper bound
\begin{align*}
	\frac{2 p^2}{T^2}\sum_{t=0}^{T-1} \sum_{s=t+1}^{T-1} \Cov_{F_A}(u(X_s),u(X_t)) \leq p^{2} \Big(1-\frac{1}{T}\Big) \Var_{F_{A}}(u(X_{0}))
\end{align*}
for the covariance term. This is a very crude upper bound as it does not decay to zero as $T \to \infty$. But, at the moment, the emphasis is on small $p$ so we will proceed with this upper bound anyway. As indicated above the choice of $v$ controls the term $p^{2}\Var_{F_{A}}(u(X_{0}))$. We conclude that the normalized variance \eqref{eq:prob:variance} of the estimator  $\widehat{q}_{T}$ is controlled by the choice of $v$ when $p$ is small.


\subsection{Ergodic properties}
As we have just seen the choice of the function $v$ controls the normalized variance of the estimator  for small $p$. The design of the MCMC sampler, on the other hand, determines the strength of the dependence in the Markov chain. Strong dependence implies slow convergence  which results in a high computational cost. The convergence rate of MCMC samplers can be analyzed within the theory of $\varphi$-irreducible Markov chains. Fundamental results for $\varphi$-irreducible Markov chains are given in  \cite{Meyn1993, Nummelin1984}. We will focus on conditions that imply a geometric convergence rate. The conditions given below are well studied in the context of MCMC samplers. Conditions for geometric ergodicity in the context of Gibbs samplers have been studied by e.g.\  \cite{Chan1993, Smith1992, Tierney1994}, and for Metropolis-Hastings algorithms by \cite{Mengersen1996}.

A Markov chain $(X_{t})_{t\geq 0}$ with transition kernel $p(x,\cdot) = \Prob(X_{t+1} \in \cdot \mid X_{t} = x)$ is $\varphi$-irreducible if there exists a measure $\varphi$ such that $\sum_{t} p^{(t)}(x,\cdot) \ll \varphi(\cdot)$, where $p^{(t)}(x, \cdot) =  \Prob(X_{t} \in \cdot \mid X_{0} = x)$ denotes the $t$-step transition kernel and $\ll$ denotes absolute continuity.
A Markov chain with invariant distribution $\pi$ it is called geometrically ergodic if there exists a positive function $M$ and a constant $r \in (0,1)$ such that
\begin{align}\label{eq:geomerg}
	\| p^{(t)}(x,\cdot) - \pi(\cdot)\|_{\text{TV}} \leq M(x) r^{t},
\end{align}
where $\|\cdot \|_{\text{TV}}$ denotes the total-variation norm.
This condition ensures that the distribution of the Markov chain converges at a geometric rate to the invariant distribution. If the function $M$ is bounded, then the Markov chain is said to be uniformly ergodic. Conditions such as \eqref{eq:geomerg} may be difficult to establish directly and are therefore substituted by suitable minorization or drift conditions. A minorization condition holds on a set $C$ if there exist a probability measure $\nu$, a positive integer $t_{0}$, and $\delta  > 0$ such that
\begin{align*}
	p^{(t_{0})}(x,B) \geq \delta \nu(B),
\end{align*}
for all $x \in C$ and Borel sets $B$. In this case $C$ is said to be a small set.
Minorization conditions have been used for obtaining rigorous bounds on the convergence of MCMC samplers, see e.g.\ \cite{Rosenthal1995}.

If the entire state space is small, then the Markov chain is uniformly ergodic.
Uniform ergodicity does typically not hold for Metropolis samplers,  \cite{Mengersen1996} Theorem 3.1. Therefore useful sufficient conditions for geometric ergodicity are often given in the form of drift conditions \cite{Chan1993, Mengersen1996}. Drift conditions are also useful for establishing central limit theorems for MCMC algorithms, see \cite{Jones2004} and the references therein. When studying simulation algorithms for random walks, in Section \ref{sec:htrw}, we will encounter Gibbs samplers that are uniformly ergodic.

 \subsection{Heuristic efficiency criteria}
To summarize, the heuristic arguments given above lead to the following desired properties of the estimator.
\begin{enumerate}
\item \emph{Rare event efficiency:} Construct an unbiased estimator $\widehat q_{T}$ of $p^{-1}$ according to \eqref{eq:hatpn} by finding a function $v$ which approximates the conditional density \eqref{eq:prob:density}. The choice of $v$ controls the normalized variance of the estimator.
\item \emph{Large sample efficiency:} Design the MCMC sampler, by finding an appropriate Gibbs sampler or a proposal density in the Metropolis-Hastings algorithm, such that the resulting Markov chain is geometrically ergodic.
\end{enumerate}

\section{The general formulation of the algorithm}
\label{sec:mcmc:exp}

In the previous section an estimator, based on Markov chain Monte Carlo, was introduced for computing the probability of a rare event.  In this section the same ideas are applied to the problem of computing an expectation. Here the setting is somewhat more general. For instance, there is no assumption that densities with respect to Lebesgue measure exist.

Let $X$ be a random variable with distribution $F$ and $h$ be a non-negative $F$-integrable function. The problem is to compute the expectation
\begin{equation*} 
\theta = \E\big[ h(X) \big] = \int h(x) dF(x) \text{.}
\end{equation*}
In the special case when $F$ has density $f$ and $h(x) = I\{x \in A\}$  this problem reduces to computing the probability in \eqref{obj:prob}.

The analogue of the conditional distribution in \eqref{eq:prob:density} is the distribution $F_{h}$ given by
\begin{equation*}
F_h(B)  = \frac{1}{\theta} \int_{B} h(x)dF(x), \quad \text{ for measurable sets $B$.}
\end{equation*}

Consider a Markov chain $(X_t)_{t \geq 0}$  having $F_h$ as its invariant distribution.
To define an estimator of $\theta^{-1}$, consider a probability distribution $V$ with $V \ll F_{h}$. Then it follows that $V \ll F$ and it is assumed that the density $dV/dF$ is known. Consider the estimator of $\zeta = \theta^{-1}$ given by
\begin{equation}
\widehat \zeta_{T} =  \frac{1}{T} \sum_{t=0}^{T-1} u(X_t) \label{eq:hatthetan}
\end{equation}
where
\begin{equation*} 
u(x) = \frac{1}{\theta} \frac{dV}{dF_h}(x).
\end{equation*}
Note that $u$ does not depend on $\theta$ because $V \ll F_{h}$ and therefore
\begin{equation*}
u(x) = \frac{1}{\theta} \frac{dV}{dF_h}(x) = \frac{1}{h(x)}\frac{dV}{dF}(x),
\end{equation*}
for $x$ such that $h(x) > 0$. The estimator \eqref{eq:hatthetan} is a generalization of the estimator \eqref{eq:hatpn} where one can think of $v$ as the  density of $V$ with respect to Lebesgue measure. An estimator of $\theta$ can then constructed as $\widehat{\theta}_{T} = \widehat{\zeta}_{T}^{-1}$.

The  variance analysis of $\widehat \zeta_{T}$ follows precisely the steps outlined above in Section \ref{sec:mcmc:prob}.
The normalized variance is 
\begin{equation} \label{eq:exp:variance}
\theta^{2} \Var_{F_h} (\widehat \zeta_{T}) = \frac{\theta^2}{T} \Var_{F_{h}} \big( u(X_{0}) \big) + \frac{2\theta^2}{T^2} \sum_{t=0}^{T-1} \sum_{s=t+1}^{T-1} \Cov_{F_{h}} \big( u(X_s), u(X_t) \big) \text{,}
\end{equation}
where the first term can be rewritten as
\begin{eqnarray*}
\frac{\theta^2}{T} \Var_{F_h} \big( u(X_{0}) \big) & = & \frac{\theta^2}{T} \Big( \E_{F_h}\big[ u(X_{0})^2 \big] - \E_{F_h} \big[u(X_{0})\big]^2 \Big) \\
& = & \frac{\theta^2}{T} \Big( \int \Big( \frac{1}{\theta} \frac{dV}{dF_h}(x) \Big)^2 F_h(dx) - \frac{1}{\theta^2} \Big) \\
& = & \frac{1}{T} \Big( \int \frac{dV}{dF_h}(x) V(dx) - 1 \Big) \\
& = & \frac{1}{T} \Big( \E_V\Big[ \frac{dV}{d F_h} \Big] - 1 \Big).
\end{eqnarray*}
The analysis above indicates that an appropriate choice of $V$ is such that
$\E_{V} [ \frac{dV}{d F_h}]$
is close to $1$. Again,  the ideal choice would be taking $V = F_{h}$ leading to zero variance. This choice is not feasible but nevertheless suggests selecting $V$ as an approximation of $F_{h}$.
The crude upper bound  for the covariance term in \eqref{eq:exp:variance} 
is valid, just as in Section \ref{sec:mcmc:prob}.

\subsection{Asymptotic efficiency criteria}
\label{sec:as.eff}

Asymptotic efficiency can be conveniently  formulated in terms of a limit criteria as a large deviation parameter tends to infinity. As is customary in problems related to rare-event simulation the problem at hand is embedded in a sequence of problems, indexed by $n = 1,2,\dots$. The general setup is formalized as follows.

Let $(X^{(n)})_{n \geq 1}$ be a sequence of random variable with $X^{(n)}$ having distribution $F^{(n)}$.  Let $h$ be a non-negative function, integrable with respect to $F^{(n)}$, for each $n$. Suppose
\begin{equation*}
\theta^{(n)} = \E\big[ h(X^{(n)}) \big] = \int h(x) dF^{(n)}(x) \to 0,
\end{equation*}
as $n\to \infty$. The problem is to compute $\theta^{(n)}$ for some large $n$.

Denote by $F_{h}^{(n)}$ the distribution with $dF_{h}^{(n)}/dF^{(n)} = h/\theta^{(n)}$. For the $n$th problem, a Markov chain $(X_{t}^{(n)})_{t=0}^{T-1}$  with invariant distribution $F_{h}^{(n)}$ is generated by an MCMC algorithm. The estimator of $\zeta^{(n)} = (\theta^{(n)})^{-1}$ is based on a probability distribution $V^{(n)}$, such that $V^{(n)}\ll F_{h}^{(n)}$, with known density with respect to $F^{(n)}$. An estimator $\widehat{\zeta}^{(n)}_{T}$ of $\zeta$ is given by
\begin{equation*}
	\widehat \zeta^{(n)}_{T} = \frac{1}{T} \sum_{t=0}^{T-1} u^{(n)}(X_t^{(n)}),
\end{equation*}
where
\begin{equation*}
u^{(n)}(x) = \frac{1}{h(x)} \frac{dV^{(n)}}{dF^{(n)}}(x).
\end{equation*}


The heuristic efficiency criteria in Sections \ref{sec:mcmc:prob} can now be rigorously formulated as follows:
\begin{enumerate}
\item \emph{Rare-event efficiency:} Select the probability distributions $V^{(n)}$ such that
\begin{align*}
	(\theta^{(n)})^{2}\Var_{F^{(n)}_{h}}(u^{(n)}(X)) \to 0, \text{ as } n\to \infty.
\end{align*}
\item \emph{Large sample size efficiency:} Design the MCMC sampler, by finding an appropriate Gibbs sampler or a proposal density for the Metropolis-Hastings algorithm, such that, for each $n \geq 1$, the Markov chain $(X_{t}^{(n)})_{t\geq 0}$ is geometrically ergodic.
\end{enumerate}


\begin{remark}
	The rare-event efficiency criteria is formulated in terms of the efficiency of estimating $(\theta^{(n)})^{-1}$ by $\widehat{\zeta}_{T}^{(n)}$. If one insists on studying the mean and variance of $\widehat{\theta}^{(n)}_{T} = (\widehat{\zeta}^{(n)}_{T})^{-1}$, then the effects of the transformation $x \mapsto x^{-1}$ must be taken into account. For instance, the estimator $\widehat{\theta}^{(n)}_{T}$ is biased and its variance could be infinite.
\end{remark}

\section{A random walk with heavy-tailed steps}
\label{sec:htrw}

In this section the estimator introduced in Section \ref{sec:mcmc:prob} is applied to compute the probability that a random walk with heavy-tailed steps exceeds a high threshold.

Let $Y_1,\ldots,Y_n$ be independent and identically distributed random variables with common distribution $F_Y$ and density $f_Y$ with respect to Lebesgue measure. Consider the random walk $S_n=Y_1+\cdots+Y_n$ and the problem of computing the probability
$$ p^{(n)} = \Prob(S_n > a_n) \text{,} $$
where $a_n \to \infty $ sufficiently fast that $p^{(n)} \rightarrow 0$ as $n \rightarrow \infty$.

It is convenient to denote by $\Ybold^{(n)}$ the $n$-dimensional random vector
\begin{equation*}
\Ybold^{(n)} = (Y_1, \ldots , Y_n)^{\trans} \text{,}
\end{equation*}
and the set
\begin{align*}
 	A_{n} = \{\by \in \R^{n}:  \bone^{\trans}\by > a_n \},
\end{align*}
where $\bone= (1,\dots, 1)^{\trans}  \in \R^{n}$ and $\by = (y_1,\ldots,y_n)^{\trans}$. With this notation
$$ p^{(n)} = \Prob(S_n > a_n)  = \Prob(\bone^{\trans} \Ybold^{(n)} > a_{n}) = \Prob(\Ybold^{(n)} \in A_{n}) \text{.} $$
The conditional distribution
\begin{equation*}
F_{A_{n}}^{(n)} (\cdot) = \Prob( \Ybold^{(n)} \in \cdot \mid \Ybold^{(n)} \in A_{n} ),
\end{equation*}
has density
\begin{eqnarray} \label{eq:rw:density}
\frac{dF^{(n)}_{A_{n}}}{dx} (y_{1}, \dots,y_{n})
& = & \frac{\prod_{j = 1}^n f_Y(y_j) I\{ y_{1}+\dots+y_{n} > a_n \}}{p^{(n)}}  \text{.}
\end{eqnarray}

The first step towards defining the estimator of $p^{(n)}$ is to construct the Markov chain $(\Ybold_t^{(n)})_{t\geq 0}$ whose invariant density is given by  \eqref{eq:rw:density} using a Gibbs sampler. In short, the Gibbs sampler updates one element of $\Ybold_{t}^{(n)}$ at a time keeping the other elements  constant. Formally the algorithm proceeds as follows.

\begin{alg} \label{alg:rw}
Start at an initial state $\Ybold_0^{(n)} = (Y_{0,1}, \dots, Y_{0,n})$ where $Y_{0,1} + \dots + Y_{0,n} > a_n$. Given $\Ybold_t^{(n)} = (Y_{t,1}, \dots, Y_{t,n})$, for some $t= 0,1,\ldots$, the next state $\Ybold_{t+1}^{(n)}$ is sampled as follows:
\begin{enumerate}
\item Draw  $j_{1}, \dots, j_{n}$  from $\{1,\dots,n\}$ without replacement and proceed by updating the components of $\Ybold_{t}^{(n)}$ in the order thus obtained.
\item For each $k=1,\dots,n$, repeat the following.
\begin{enumerate}
\item Let $j = j_{k}$ be the index to be updated and write $$\Ybold_{t,-j} = (Y_{t,1}, \dots, Y_{t,j-1},Y_{t,j+1},\dots,Y_{t,n}).$$ Sample $Y_{t,j}'$ from the conditional distribution of $Y$ given that the sum exceeds the threshold. That is,
    $$ \Prob(Y_{t,j}' \in \cdot \mid \Ybold_{t,-j}) = \Prob\Big( Y \in \cdot \mid  Y + \sum_{k \neq j} Y_{t,k} > a_n \Big) \text{.} $$
\item Put $\Ybold_{t}' = (Y_{t,1}, \dots, Y_{t,j-1},Y_{t,j}',Y_{t,j+1}, \dots,Y_{t,n})^{\trans}$.
\end{enumerate}
\item Draw a random permutation $\pi$ of the numbers $\{1,\dots,n\}$ from the uniform distribution and put $\Ybold_{t+1}^{(n)} = (Y_{t,\pi(1)}', \dots, Y_{t,\pi(n)}')$.
\end{enumerate}
Iterate steps (1)-(3) until the entire Markov Chain $(\Ybold_t^{(n)})_{t=0}^{T-1}$ is constructed.
\end{alg}
\begin{remark}
In the heavy-tailed setting the trajectories of the random walk leading to the rare event are likely to consist of one large increment (the big jump) while the other increments are average. The purpose of the permutation step is to force the Markov chain to mix faster by moving the big jump to different locations. However, the permutation step in Algorithm \ref{alg:rw} is not really needed when considering the probability $\Prob(S_{n} > a_{n})$.  This is due to the fact that the summation  is invariant of the ordering of the steps.
\end{remark}

The following proposition confirms that the Markov chain $(\Ybold_t^{(n)})_{t\geq 0}$, generated by Algorithm \ref{alg:rw}, has $F_{A_{n}}^{(n)}$ as its invariant distribution.
\begin{proposition}
	The Markov chain $(\Ybold_{t}^{(n)})_{t\geq 0}$, generated by Algorithm \ref{alg:rw},  has the conditional distribution $F_{A_{n}}^{(n)}$ as its invariant distribution.
\end{proposition}

\begin{proof}
The goal is to show that each updating step (Step 2 and 3) of the algorithm preserves stationarity. Since the conditional distribution $F_{A_{n}}^{(n)}$ is permutation invariant it is clear that Step 3 preserves stationarity. Therefore it is sufficient to consider Step 2 of the algorithm.

Let $P_{j}(\ybold,\cdot)$ denote the transition probability of the Markov chain $(\Ybold_{t}^{(n)})_{t\geq 0}$ corresponding to the $j$th component being updated.  It is sufficient to show that, for all $j =1,\dots,m$ and all Borel sets of product form $B_{1}\times \dots \times B_{n} \subset A_{n}$, the following equality holds:
\begin{align*}
	F_{A_{n}}^{(n)}(B_{1}\times \dots \times B_{n}) &= \E_{F_{A_{n}}^{(n)}}[P_{j}(\Ybold, B_{1}\times \dots \times B_{n})].
\end{align*}
Observe that, because $B_{1}\times \dots \times B_{n} \subset A_{n}$,
\begin{align*}
	&F_{A_{n}}^{(n)}(B_{1}\times \dots \times B_{n}) = \E\Big[\prod_{k =1}^{n} I\{Y_{k} \in B_{k}\}\mid Y_{1}+\dots + Y_{n} > a_{n}\Big]\\
	&\; =  \frac{\E[I\{Y_{j} \in B_{j}\} I\{Y_{1}+\dots + Y_{n} > a_{n}\} \prod_{k \neq j}^{m} I\{Y_{k} \in B_{k}\}]}{\Prob(Y_{1}+\dots + Y_{n} > a_{n})}
	\\
	&\; =  \frac{\E\Big[\frac{\E[I\{Y_{j} \in B_{j}\} \mid Y_{j} > a_{n}-(Y_{1}+\dots+Y_{j-1}+Y_{j+1}+\dots Y_{n}),Y_{1},\dots,Y_{j-1},Y_{j+1},\dots,Y_{n}]\prod_{k \neq j}^{n} I\{Y_{k} \in B_{k}\}}{\Prob(Y_{j}> a_{n}-(Y_{1}+\dots+Y_{j-1}+Y_{j+1}+\dots Y_{n}) \mid Y_{1},\dots,Y_{j-1},Y_{j+1},\dots,Y_{n})}\Big]}{\Prob(Y_{1}+\dots + Y_{n} > a_{n})}
	\\
&\; =	\frac{\E[P_{j}((Y_{1},\dots,Y_{n})^{\trans}, B_{1}\times \dots \times B_{n})\prod_{k \neq j}^{n} I\{Y_{k} \in B_{k}\}]}{P(Y_{1}+\dots + Y_{n} > a_{n})}
	 \\ &\quad = \E[P_{j}((Y_{1},\dots,Y_{n})^{\trans}, B_{1}\times \dots \times B_{n}) \mid Y_{1}+\dots + Y_{n} > a_{n}]\\
	 &\; =  \E_{F_{A_{n}}^{(n)}}[P_{j}(\Ybold, B_{1}\times \dots \times B_{n})].
\end{align*}
\end{proof}

As for the ergodic properties, Algorithm \ref{alg:rw} produces a Markov chain which is uniformly ergodic.

\begin{proposition}\label{prop:erg}
For each $n \geq 1$, the Markov chain $(\Ybold_{t}^{(n)})_{t\geq 0}$ is uniformly ergodic. It satisfies the following minorization condition:
there exists $\delta > 0$ such that
\begin{align*}
	\Prob(\Ybold_{1}^{(n)} \in B \mid \Ybold_{0}^{(n)} = \ybold) \geq \delta F_{A_{n}}^{(n)}(B),
\end{align*}
for all $\ybold \in A_{n}$ and all Borel sets $B \subset A_{n}$.
\end{proposition}

\begin{proof}
	Take an arbitrary $n \geq 1$. Uniform ergodicity can be deduced from the following minorization condition (see \cite{Nummelin1984}): there exists a probability measure $\nu$, $\delta > 0$, and an integer $t_{0}$ such that
	\begin{align*}
	\Prob(\Ybold_{t_{0}}^{(n)} \in B \mid \Ybold_{0}^{(n)} = \ybold) \geq \delta \nu(B),
\end{align*}
for every $\ybold \in A_{n}$ and  Borel set $B \subset A_{n}$.  Take $\ybold \in A_{n}$ and write $g(\,\cdot \mid \ybold)$ for the density of $\Prob(\Ybold^{(n)}_{1} \in \cdot \mid \Ybold_{0}^{(n)} = \by)$.
The goal is to show that the minorization condition holds with $t_{0} = 1$, $\delta = p^{(n)}/n!$, and $\nu = F_{A_{n}}^{(n)}$.

For any $\xbold \in A_{n}$ there exists an ordering $j_{1}, \dots, j_{n}$ of the numbers $\{1, \dots, n\}$ such that
\begin{align*}
	y_{j_{1}} \leq x_{j_{1}}, \dots, y_{j_{k}} \leq x_{j_{k}}, y_{j_{k+1}} > x_{j_{k+1}}, \dots y_{j_{n}} > x_{j_{n}},
\end{align*}
for some $k \in \{0,\dots, n\}$. The probability to draw this particular ordering in Step 1 of the algorithm is at least $1/n!$. It follows that
\begin{align*}
	g(\xbold \mid \ybold) &\geq \frac{1}{n!} \frac{f_{Y}(x_{j_{1}})I\{x_{j_{1}} \geq a_{n}-\sum_{i\neq j_{1}}y_{i}\}}{\overline{F}_{Y}(a_{n}-\sum_{i\neq j_{1}}y_{i})} \\ & \quad \times \frac{f_{Y}(x_{j_{2}})I\{x_{j_{2}} \geq a_{n}-\sum_{i\neq j_{1},j_{2}}y_{i} - x_{j_{1}}\}}{\overline{F}_{Y}(a_{n}-\sum_{i\neq j_{1},j_{2}}y_{i} - x_{j_{1}})} \\ & \quad \; \;\vdots \\ & \quad  \times \frac{f_{Y}(x_{j_{n}})I\{x_{j_{n}} \geq a_{n}-x_{j_{1}}-\dots x_{j_{n-1}}\}}{\overline{F}_{Y}(a_{n}-x_{j_{1}}-\dots x_{j_{n-1}})}.
\end{align*}
By construction of the ordering $j_{1}, \dots, j_{n}$ all the indicators are equal to $1$ and the expression in the last display is bounded from below by
\begin{align*}
	\frac{1}{n!} \prod_{j=1}^{n} f_{Y}(x_{j}) = \frac{p^{(n)}}{n!} \cdot \frac{\prod_{j=1}^{n} f_{Y}(x_{j}) I\{x_{1} + \dots + x_{n} > a_{n}\}}{p^{(n)}}.
\end{align*}
The proof is completed by integrating both sides of the inequality over any Borel set $B \subset A_{n}$.
\end{proof}

Note that so far the distributional assumption of steps $Y_{1}, \dots, Y_{n}$ of the random walk have been very general. The only assumption has been the existence of a density. For the rare-event properties of the estimator the design of $V^{(n)}$ is essential and this is where the distributional assumptions become important. In this section a heavy-tailed random walk is considered. To be precise, assume that the variables $Y_{1}, \dots, Y_{n}$ are nonnegative and that the tail of $F_{Y}$ is heavy in the sense that there is a sequence $(a_{n})$ of real numbers such that
\begin{align}\label{eq:heavytail}
	\lim_{n \to \infty}\frac{\Prob(S_{n}> a_{n})}{\Prob(M_{n}> a_{n})} = 1,
\end{align}
where $M_{n}$ denotes the maximum of $Y_{1}, \dots, Y_{n}$. The class of distributions for which \eqref{eq:heavytail} holds is large and includes the subexponential distributions. General conditions on the sequence $(a_{n})$ for which \eqref{eq:heavytail} holds are given in \cite{DDS08}. For instance, if $\overline{F}_Y$ is regularly varying at $\infty$ with index $\beta > 1$ then (\ref{eq:heavytail}) holds with $a_{n} = a n$, for $a > 0$.

Next consider the choice of $V^{(n)}$. As observed in Section \ref{sec:mcmc:prob} a good approximation to the conditional distribution $F_{A_{n}}^{(n)}$ is a candidate for $V^{(n)}$. For a heavy-tailed random walk the ``one big jump'' heuristics says that the sum is large most likely  because one of the steps is large. Based on the assumption \eqref{eq:heavytail}  a good candidate for $V^{(n)}$ is the conditional distribution,
$$ V^{(n)}(\cdot) = \Prob( \Ybold^{(n)} \in \cdot \mid M_n > a_n).$$
Then $V^{(n)}$ has a known density with respect to $F^{(n)}(\cdot) = \Prob(\Ybold^{(n)} \in \cdot )$ given by
\begin{align*}
	\frac{d V^{(n)}}{dF^{(n)}}(\ybold) = \frac{1}{\Prob(M_{n} > a_{n})} I\{\ybold : \vee_{j=1}^{n} y_{j} > a_{n}\} =   \frac{1}{1-F_{Y}(a_{n})^{n}} I\{\ybold : \vee_{j=1}^{n}y_{j} > a_{n}\}.
\end{align*}

The estimator of $q^{(n)} = \Prob(S_{n} > a_{n})^{-1}$ is then given by
\begin{equation}\label{eq:rwest}
\widehat{q}^{(n)}_{T} = \frac{1}{T} \sum_{t=0}^{T-1}\frac{d V^{(n)}}{dF^{(n)}}(\Ybold_t^{(n)}) =
\frac{1}{1-F_{Y}(a_{n})^{n}} \cdot \frac{1}{T} \sum_{t=0}^{T-1} I\{\vee_{j=1}^{n} Y_{t,j} > a_{n}\}
\end{equation}
where $(\Ybold_{t}^{(n)})_{t\geq 0}$ is generated by Algorithm \ref{alg:rw}. Note that the estimator \eqref{eq:rwest} can be viewed as the asymptotic approximation $(1-F_{Y}(a_{n})^{n})^{-1}$ of $(p^{(n)})^{-1}$ multiplied by the random correction factor $\frac{1}{T} \sum_{t=0}^{T-1} I\{\vee_{j=1}^{n} Y_{t,j} > a_{n}\}$. The efficiency of this estimator is based on the fact that the random correction factor is likely to be close to $1$ and has small variance.

\begin{theorem} \label{thm:vanishing}
Suppose that \eqref{eq:heavytail} holds. Then the estimator $\widehat q_{T}^{(n)}$ in \eqref{eq:rwest} has vanishing normalized variance for estimating $(p^{(n)})^{-1}$. That is,
\begin{align*}
	\lim_{n \to \infty}  (p^{(n)})^{2} \Var_{F^{(n)}_{A_{n}}}(\widehat q_{T}^{(n)})= 0.
\end{align*}
\end{theorem}

\begin{proof}
With $u^{(n)}(\ybold) = \frac{1}{1-F_{Y}(a_{n})^{n}} I\{\vee_{j=1}^{n} y_{j} > a_{n}\}$ it follows from  \eqref{eq:heavytail} that
\begin{align*}
	 &(p^{(n)})^{2} \Var_{F_{A_{n}}^{(n)}}(u^{(n)}(\Ybold^{(n)})) \\
	 & \quad = \frac{\Prob(S_{n}> a_{n})^{2}}{\Prob(M_{n} > a_{n})^{2}} \Var_{F_{A_{n}}^{(n)}}(I\{\Ybold : \vee_{j=1}^{n} Y_{j} > a_{n}\})\\
	 & \quad =  \frac{\Prob(S_{n}> a_{n})^{2}}{\Prob(M_{n} > a_{n})^{2}} \Prob(M_{n} > a_{n} \mid S_{n} > a_{n})\Prob(M_{n} \leq a_{n} \mid S_{n} > a_{n}) \\
	 & \quad = \frac{\Prob(S_{n}> a_{n})}{\Prob(M_{n} > a_{n})} \Big(1-\frac{\Prob(M_{n}> a_{n})}{\Prob(S_{n} > a_{n})}\Big) \to 0.
\end{align*}
This completes the proof.
\end{proof}

\begin{remark}
	Theorem \ref{thm:vanishing} covers a wide range of heavy-tailed distributions and even allows the number of steps to increase with $n$. Its proof is elementary. This is in sharp contrast to the existing proofs of efficiency (bounded relative error, say) for importance sampling algorithms that cover less general models and tend to be long and technical, see e.g.\ \cite{Dupuis2007, BlanchetLiu2008, BL10}. It must be mentioned, though, that Theorem \ref{thm:vanishing} proves efficiency for computing $(p^{(n)})^{-1}$, whereas the authors of \cite{Dupuis2007, BlanchetLiu2008, BL10} prove efficiency for a direct computation of $p^{(n)}$.
\end{remark}

\subsection{An extension to random sums}
\label{sec:geometric}

In application to queueing and ruin theory there is particular interest in sums consisting of a random number of heavy-tailed steps. For instance, the stationary distribution of the waiting time and the workload of an $M/G/1$ queue can be represented as a random sum, see Asmussen (2003), Theorem 5.7. The classical Cram\'er-Lundberg model for the
total claim amount faced by an insurance company is another standard example of a random sum. In this section Algorithm \ref{alg:rw} is modified to efficiently estimate hitting probabilities for heavy-tailed random sums.

Let $Y_1, Y_{2}, \dots$ be non-negative independent random variables with common distribution $F_Y$ and density $f_Y$. Let $(N^{(n)})_{n\geq 1}$ be integer valued random variables independent of $Y_1, Y_{2}, \dots$. Consider the random sum $S_{N^{(n)}} = Y_1 + \cdots + Y_{N^{(n)}}$ and the problem of computing the probability
$$ p^{(n)} = \Prob(S_{N^{(n)}} > a_n) \text{,} $$
where $a_n \to \infty$ at an appropriate rate.

Denote by $\overline{\Ybold}^{(n)}$ the vector $(N^{(n)}, Y_{1}, \dots, Y_{N^{(n)}})^{\trans}$. The conditional distribution of $\overline \Ybold^{(n)}$ given $S_{N^{(n)}} > a_{n}$ is given by
\begin{align*}
	&\Prob(N^{(n)} = k, (Y_{1}, \dots, Y_{k}) \in \cdot \mid S_{N^{(n)}} > a_n) \\ & \quad = \frac{\Prob((Y_{1}, \dots, Y_{k}) \in \cdot\,, S_{k} > a_{n})\Prob(N^{(n)} = k)}{p^{(n)}}.
\end{align*}

A Gibbs sampler for sampling from the above conditional distribution can be constructed essentially as in Algorithm \ref{alg:rw}. The only additional difficulty is to update the random number of steps in an appropriate way. In the following algorithm a particular distribution for updating the number of steps is proposed. To ease the notation the superscript $n$ is suppressed in the description of the algorithm.

\begin{alg} \label{alg:random}
To initiate, draw  $N_{0}$ from $\Prob(N \in \cdot)$ and $Y_{0,1}, \dots, Y_{0,N_{0}}$ such that $Y_{0,1} + \dots + Y_{0,N_{0}} > a_{n}$.  Each iteration of the algorithm consists of the following steps. Suppose $\overline \Ybold_{t} = (k_{t} ,y_{t,1}, \dots, y_{t,k_{t}})$ with $y_{t,1}+\dots+y_{t,k_{t}} > a_{n}$. Write $k^{*}_{t} = \min\{j : y_{t,1}+\dots+y_{t,j} > a_{n}\}$.
\begin{enumerate}
\item Sample number of steps $N_{t+1}$ from the distribution
\begin{align*}
	p(k_{t+1} \mid k^{*}_{t}) = \frac{\Prob(N = k_{t+1}) I\{k_{t+1} \geq k^{*}_{t}\}}{P(N \geq k^{*}_{t})}.
\end{align*}
If $N_{t+1} > k_{t}$, sample $Y_{t+1,k_{t}+1}, \dots, Y_{t+1,N_{t+1}}$ independently from $F_{Y}$ and put $\Ybold_{t}^{(1)} = (Y_{t,1}, \dots, Y_{t,k_{t}}, Y_{t+1,k_{t}+1}, \dots, Y_{t+1,N_{t+1}})$.
\item Proceed by updating all the individual steps as in Algorithm \ref{alg:rw}.
\begin{enumerate}
\item Draw  $j_{1}, \dots, j_{N_{t+1}}$  from $\{1,\dots,N_{t+1}\}$ without replacement and proceed by updating the components of $\Ybold_{t}^{(1)}$ in the order thus obtained.
\item For each $k=1,\dots,N_{t+1}$, repeat the following.
\begin{enumerate}
\item Let $j = j_{k}$ be the index to be updated and write $$\Ybold_{t,-j}^{(1)} = (Y_{t,1}^{(1)}, \dots, Y_{t,j-1}^{(1)},Y_{t,j+1}^{(1)},\dots,Y_{t,N_{t+1}}^{(1)}).$$ Sample $Y_{t,j}^{(2)}$ from the conditional distribution of $Y$ given that the sum exceeds the threshold. That is,
    $$ \Prob(Y_{t,j}^{(2)} \in \cdot \mid \Ybold_{t,-j}^{(1)}) = \Prob\Big( Y \in \cdot \mid  Y + \sum_{k\neq j} Y_{t,k}^{(1)} > a_n \Big) \text{.} $$
\item Put $\Ybold_{t}^{(2)} = (Y_{t,1}^{(1)}, \dots, Y_{t,j-1}^{(1)},Y_{t,j}^{(2)},Y_{t,j+1}^{(1)}, \dots,Y_{t,N_{t+1}}^{(1)})^{\trans}$.
\end{enumerate}
\item Draw a random permutation $\pi$ of the numbers $\{1,\dots,N_{t+1}\}$ from the uniform distribution and put $\overline{\Ybold}_{t+1} = (N_{t+1}, Y_{t,\pi(1)}^{(2)}, \dots, Y_{t,\pi(N_{t+1})}^{(2)})$.
\end{enumerate}
\end{enumerate}
Iterate until the entire Markov Chain $(\overline{\Ybold}_t)_{t=0}^{T-1}$ is constructed.
\end{alg}

\begin{proposition}
 The Markov chain {\rm $(\overline{\Ybold}_{t})_{t \geq 0}$} generated by Algorithm \ref{alg:random} has the conditional distribution $\Prob((N, Y_{1}, \dots, Y_{N}) \in \cdot \mid Y_{1} + \dots Y_{N} > a_{n})$ as its invariant distribution.
\end{proposition}

\begin{proof}
The only essential difference from Algorithm \ref{alg:rw} is the first step of the algorithm, where the number of steps and possibly the additional steps are updated. Therefore, it is sufficient to prove that the first step of the algorithm preserves stationarity. The transition probability of the first step, starting from a state $(k_{t}, y_{t,1}, \dots, y_{t,k_{t}})$ with $k^{*}_{t} = \min\{j : y_{t,1}+\dots+y_{t,j} > a_{n}\}$, can be written as follows.
\begin{align*}
& P^{(1)}(k_{t},y_{t,1}, \dots, y_{t, k_{t}}; k_{t+1}, A_{1}\times \cdots \times A_{k_{t+1}}) \\
& \, = \Prob\big( N_{t+1} = k_{t+1}, (Y_{t,1}, \dots, Y_{t,k_{t+1}}) \in A_{1} \times \dots \times A_{k_{t+1}} \\
& \quad \quad \quad \quad \mid N_{t} = k_{t}, Y_{t,1} = y_{t,1}, \dots, Y_{t,k_{t}}=y_{t,k_{t}} \big) \\
& \, = \left\{\begin{array}{ll}
    p(k_{t+1} \mid k^{*}_{t}) \prod_{k=1}^{k_{t+1}} I\{y_{t,k} \in A_{k}\}, & k_{t+1} \leq k_{t},\\
    p(k_{t+1} \mid k^{*}_{t}) \prod_{k=1}^{k_{t}} I\{y_{t,k} \in A_{k}\} \prod_{k=k_{t}+1}^{k_{t+1}} F_{Y}(A_{k}),
    & k_{t+1} > k_{t}.\end{array}\right.
\end{align*}
Consider the stationary probability of a set of the form $\{k_{t+1}\}\times A_{1}\times \cdots \times A_{k_{t+1}}$. It holds that
\begin{align*}
& \E_{\pi}[P^{(1)}(N_{t},Y_{t,1}, \dots, Y_{t,N_{t}}; k_{t+1}, A_{1}\times \cdots \times A_{k_{t+1}})] \\
& \, = \frac{1}{\Prob(S_{N} > a_{n})} \E[ P^{(1)}(N,Y_{1}, \dots, Y_{N}; k_{t+1}, A_{1}\times \cdots \times A_{k_{t+1}})I\{S_N > a_{n}\}] \\
& \, = \frac{1}{\Prob(S_{N} > a_{n})}\sum_{k_{t}=1}^{\infty} \Prob(N = k_{t}) \\
& \quad \quad \quad \quad \quad \quad \times P^{(1)}(k_{t},Y_{1}, \dots, Y_{k_t}; k_{t+1}, A_{1}\times \cdots \times A_{k_{t+1}}) I\{S_{k_{t}} > a_{n}\}.
\end{align*}
With $B_{k^{*}} = \{(y_{1}, y_{2}, \dots) : \min\{j: y_{1}+\dots+y_{j} > a\} = k^{*}\}$,  $A^{\otimes}_{k_{t}} = A_{1}\times \cdots \times A_{k_{t}}$, and $A^{\otimes}_{k_{t+1}} = A_{1}\times \cdots \times A_{k_{t+1}}$ the expression in the last display can be written as
\begin{align*}
&\frac{1}{\Prob(S_{N} > a_{n})}\Bigg( \sum_{k_{t}=1}^{k_{t+1}} \Prob(N = k_{t}) \\
& \quad \quad \quad \times \E\Big[ \sum_{k^{*} =1}^{k_{t}} I\{(Y_{1}, \dots, Y_{k_{t}}) \in B_{k^{*}}\} P^{(1)}(k_{t},Y_{1}, \dots, Y_{k_{t}}; k_{t+1}, A^{\otimes}_{k_{t+1}}) \Big] \\
& \quad + \sum_{k_{t}=k_{t+1}+1}^{\infty} \Prob(N = k_{t}) \\
& \quad \quad \quad \times \E\Big[ \sum_{k^{*} =1}^{k_{t+1}} I\{(Y_{1}, \dots, Y_{k_{t+1}}) \in B_{k^{*}}\} P^{(1)}(k_{t},Y_{1}, \dots, Y_{k_{t}}; k_{t+1}, A^{\otimes}_{k_{t+1}}) \Big] \Bigg).
\end{align*}
Inserting the expression for $P^{(1)}$ the last expression equals
\begin{align*}
&\frac{1}{\Prob(S_{N} > a)} \Bigg( \sum_{k_{t}=1}^{k_{t+1}} \Prob(N = k_{t}) \\
& \quad \quad \quad \times \sum_{k^{*} =1}^{k_{t}} \Prob\big( (Y_{1}, \dots, Y_{k_{t}}) \in B_{k^{*}} \cap A^{\otimes}_{k_{t}} \big) p(k_{t+1} \mid k^{*}) \prod_{j=k_{t}+1}^{k_{t+1}} F_{Y}(A_{j}) \\
& \quad + \sum_{k_{t}=k_{t+1}+1}^{\infty} \Prob(N = k_{t})
\sum_{k^{*} =1}^{k_{t+1}} \Prob\big( (Y_{1}, \dots, Y_{k_{t+1}}) \in B_{k^{*}} \cap A^{\otimes}_{k_{t+1}} \big) p(k_{t+1} \mid k^{*}) \Bigg) \text{.}
\end{align*}
Changing the order of summation the last expression equals
\begin{align*}
& \frac{1}{\Prob(S_{N} > a_{n})} \Bigg( \sum_{k^{*}=1}^{k_{t+1}} \sum_{k_{t}=k^{*}}^{k_{t+1}} \Prob(N = k_{t})
\\ & \quad \quad \quad \times
\Prob \big((Y_{1}, \dots, Y_{k_{t}}) \in B_{k^{*}} \cap A^{\otimes}_{k_{t}} \big) p(k_{t+1} \mid k^{*}) \prod_{j=k_{t}+1}^{k_{t+1}} F_{Y}(A_{j}) \\
& \quad + \sum_{k^{*}=1}^{k_{t+1}} \sum_{k_{t}=k_{t+1}+1}^{\infty} \Prob(N = k_{t})
\Prob \big( (Y_{1}, \dots, Y_{k_{t+1}}) \in B_{k^{*}}  \cap A^{\otimes}_{k_{t+1}} \big) p(k_{t+1} \mid k^{*})\Bigg) \text{.}
\end{align*}
Since $\Prob \big((Y_{1}, \dots, Y_{k_{t}}) \in B_{k^{*}} \cap A^{\otimes}_{k_{t}} \big) \prod_{j=k_{t}+1}^{k_{t+1}} F_{Y}(A_{j}) = \Prob \big( (Y_{1}, \dots, Y_{k_{t+1}}) \in B_{k^{*}} \cap A^{\otimes}_{k_{t+1}} \big)$ the last expression equals
\begin{align*}
& \frac{1}{\Prob(S_{N} > a_{n})} \Bigg( \sum_{k^{*}=1}^{k_{t+1}} \sum_{k_{t}=k^{*}}^{k_{t+1}} \Prob(N = k_{t})
\Prob \big( (Y_{1}, \dots, Y_{k_{t+1}}) \in B_{k^{*}} \cap A^{\otimes}_{k_{t+1}} \big) p(k_{t+1} \mid k^{*}) \\
& \quad + \sum_{k^{*}=1}^{k_{t+1}} \sum_{k_{t}=k_{t+1}+1}^{\infty} \Prob(N = k_{t})
\Prob \big( (Y_{1}, \dots, Y_{k_{t+1}}) \in B_{k^{*}}  \cap A^{\otimes}_{k_{t+1}} \big) p(k_{t+1} \mid k^{*})\Bigg) \text{.}
\end{align*}
Summing over $k_{t}$ the last expression equals
\begin{align*}
& \frac{1}{\Prob(S_{N} > a_{n})} \Bigg(\sum_{k^{*}=1}^{k_{t+1}} \Prob \big((Y_{1}, \dots, Y_{k_{t+1}}) \in B_{k^{*}} \cap A^{\otimes}_{k_{t+1}} \big)
p(k_{t+1} \mid k^{*}) \Prob(k^{*} \leq N \leq k_{t+1})
\\ & \quad
+ \sum_{k^{*}=1}^{k_{t+1}} \Prob \big((Y_{1}, \dots, Y_{k_{t+1}}) \in B_{k^{*}} \cap A^{\otimes}_{k_{t+1}} \big)
p(k_{t+1} \mid k^{*}) \Prob(N \geq k_{t+1}+1)\Bigg) \text{.}
\end{align*}
From the definition of $p(k_{t+1} \mid k^{*})$ it follows that the last expression equals
\begin{align*}
& \frac{1}{\Prob(S_{N} > a_{n})} \sum_{k^{*}=1}^{k_{t+1}} \Prob\big( (Y_{1}, \dots, Y_{k_{t+1}}) \in B_{k^{*}} \cap A^{\otimes}_{k_{t+1}} \big)
p(k_{t+1} \mid k^{*}) P(N \geq k^{*}) \\
& \, = \frac{1}{\Prob(S_{N} > a_{n})}\sum_{k^{*}=1}^{k_{t+1}} \Prob\big( (Y_{1}, \dots, Y_{k_{t+1}}) \in B_{k^{*}} \cap A^{\otimes}_{k_{t+1}} \big) P(N = k_{t+1}) \\
& \, = \frac{1}{\Prob(S_{N} > a_{n})} \Prob\big( (Y_{1}, \dots, Y_{k_{t+1}}) \in A^{\otimes}_{k_{t+1}} \big) P(N = k_{t+1}) \\
& \, = \Prob \big( N = k_{t+1}, (Y_{1}, \dots, Y_{k_{t+1}}) \in A^{\otimes}_{k_{t+1}} \mid Y_{1}+ \dots + Y_{N} > a_{n} \big) \text{,}
\end{align*}
which is the desired invariant distribution. This completes the proof.
\end{proof}

\begin{proposition}
The Markov chain $(\overline \Ybold_{t})_{t\geq 0}$ generated by Algorithm \ref{alg:random} is uniformly ergodic. It satisfies the following minorization condition:
there exists $\delta > 0$ such that
\begin{align*}
	\Prob(\overline \Ybold_{1} \in B \mid \overline \Ybold_{0} = \overline \ybold) \geq \delta \,\Prob((N,Y_{1}, \dots, Y_{N}) \in B \mid Y_{1}+\dots+Y_{N} > a_{n}),
\end{align*}
for all $\overline \ybold \in A = \cup_{k \geq 1} \{(k, y_{1}, \dots, y_{k}): y_{1}+ \dots + y_{k} > a_{n}\}$ and all Borel sets $B \subset A$.
\end{proposition}

The proof requires only a minor modification from the non-random case, Proposition \ref{prop:erg},  and is therefore omitted.

Next consider the distributional assumptions and the design of $V^{(n)}$. The main focus is on the rare event properties of the estimator and therefore the large deviation parameter $n$ will be suppressed to ease notation. Let the distribution of the number of steps $\Prob(N^{(n)} \in \cdot)$ to depend on $n$. By a similar reasoning as in the case of non-random number of steps the following assumption are imposed: the variables $N^{(n)}$, $Y_{1}, Y_{2}, \dots$ and the numbers $a_{n}$ are such that
\begin{align}\label{eq:heavytailrandsum}
	\lim_{n\to\infty} \frac{\Prob(Y_{1} + \dots + Y_{N^{(n)}} > a_{n})}{\Prob(M_{N^{(n)}} < a_{n})} = 1,
\end{align}
where $M_{k} = \max\{Y_{1}, \dots, Y_{k}\}$. Note that the denominator can be expressed  as
\begin{align*}
	\Prob(M_{N^{(n)}} > a_{n}) &= \sum_{k=1}^{\infty} \Prob(M_{k} > a_{n})\Prob(N^{(n)} = k)\\
	&= \sum_{k=1}^{\infty} [1-F_{Y}(a_{n})^{k}]\Prob(N^{(n)} = k) \\
	&= 1-g_{N^{(n)}}(F_{Y}(a_{n})),
\end{align*}
where $g_{N^{(n)}}(t) = \E[t^{N^{(n)}}]$ is the generating function of $N^{(n)}$. Sufficient conditions for \eqref{eq:heavytailrandsum} to hold are given in \cite{KM97}, Theorem 3.1. For instance, if $\overline{F}_{Y}$ is regularly varying at $\infty$ with index $\beta > 1$ and $N^{(n)}$ has Poisson distribution with mean $\lambda_{n} \to \infty$, as $n \to \infty$, then \eqref{eq:heavytailrandsum} holds with $a_{n} = a \lambda_{n}$, for $a > 0$.

Similarly to the non-random setting a good candidate for $V^{(n)}$ is the conditional distribution,
$$ V^{(n)}(\cdot) = \Prob( \overline{\Ybold}^{(n)} \in \cdot \mid M_{N^{(n)}} > a_n).$$
Then $V^{(n)}$ has a known density with respect to $F^{(n)}(\cdot) = \Prob(\overline{\Ybold}^{(n)} \in \cdot )$ given by
\begin{align*}
	\frac{dV^{(n)}}{dF^{(n)}}(k,y_{1}, \dots,y_{k}) &= \frac{1}{\Prob(M_{N^{(n)}} > a_{n})} I\{(y_{1}, \dots, y_{k}) : \vee_{j=1}^{k} y_{j} > a_{n}\}\\ & =   \frac{1}{1-g_{N^{(n)}}(F_{Y}(a_{n}))} I\{(y_{1}, \dots, y_{k}): \vee_{j=1}^{k}y_{j} > a_{n}\}.
\end{align*}

The estimator of $q^{(n)} = \Prob(S_{n} > a_{n})^{-1}$ is given by
\begin{equation}\label{eq:estrandsum}
\widehat q_{T}^{(n)} = \frac{1}{T} \sum_{t=0}^{T-1} \frac{dV^{(n)}}{dF^{(n)}}(\overline{\Ybold}_t^{(n)}) = \frac{1}{g_{N^{(n)}}(F_{Y}(a_{n}))}\cdot \frac{1}{T} \sum_{t=0}^{T-1} I\{\vee_{j=1}^{N_{t}} Y_{t,j} > a_{n}\},
\end{equation}
where $(\overline{\Ybold}_{t}^{(n)})_{t\geq 0}$ is generated by Algorithm \ref{alg:random}.

\begin{theorem}\label{thm:randsum}
	Suppose \eqref{eq:heavytailrandsum} holds. The estimator $\widehat{q}_{T}^{(n)}$ in \eqref{eq:estrandsum} has vanishing normalized variance. That is,
\begin{align*}
	\lim_{n \to \infty}  (p^{(n)})^{2} \Var_{\pi_{n}}(\widehat{q}^{(n)}_{T}) = 0,
\end{align*}
where $\pi_{n}$ denotes the conditional distribution $\Prob( \overline{\Ybold}^{(n)} \in \cdot \mid S_{N^{(n)}} > a_n)$.
\end{theorem}

\begin{remark}
Because the distribution of $N^{(n)}$ may depend on $n$ Theorem \ref{thm:randsum} covers a wider range of models for random sums than those studied in \cite{Dupuis2007, BL10} where the authors present provably efficient importance sampling algorithms.
\end{remark}

\begin{proof}
Since $p^{(n)} = \Prob(S_{N^{(n)}} > a_{n})$ and
\begin{align*}
	u^{(n)}(k,y_{1}, \dots, y_{k}) = \frac{I\{ \vee_{j=1}^{k} y_{j} > a_{n}\}}{\Prob(M_{N^{(n)}} > a_{n})},
\end{align*}
it follows that
\begin{align*}
	 &[p^{(n)}]^{2} \Var_{\pi_{n}}(u^{(n)}(\overline{\Ybold}^{(n)})) \\
	 & \quad = \frac{\Prob(S_{N^{(n)}}> a_{n})^{2}}{\Prob(M_{N^{(n)}} > a_{n})^{2}} \Var_{\pi_{n}}(I\{\vee_{j=1}^{N^{(n)}} Y_{j} > a_{n}\})\\
	 & \quad =  \frac{\Prob(S_{N^{(n)}}> a_{n})^{2}}{\Prob(M_{N^{(n)}} > a_{n})^{2}} \Prob(M_{N^{(n)}} > a_{n} \mid S_{N^{(n)}} > a_{n}) \Prob(M_{N^{(n)}} \leq a_{n} \mid S_{N^{(n)}} > a_{n}) \\
	 & \quad = \frac{\Prob(S_{N^{(n)}}> a_{n})}{\Prob(M_{N^{(n)}} > a_{n})} \Big(1-\frac{\Prob(M_{N^{(n)}}> a_{n})}{\Prob(S_{N^{(n)}} > a_{n})}\Big) \to 0,
\end{align*}
by \eqref{eq:heavytailrandsum}. This completes the proof.
\end{proof}

\section{Numerical experiments}
\label{sec:numeric}

In this section the performance of the estimator $\widehat{p}_{T}^{(n)} = (\widehat{q}^{(n)}_{T})^{-1}$ with $\widehat{q}^{(n)}_{T}$ as in \eqref{eq:rwest} is illustrated numerically. The literature includes numerical comparison  for many of the existing algorithms. In particular in the setting of random sums. Numerical results for the algorithms by Dupuis et al. \cite{Dupuis2007}, the hazard rate twisting algorithm by Juneja and Shahabuddin \cite{Juneja2002}, and the conditional Monte Carlo algorithm by Asmussen and Kroese \cite{Asmussen2006} can be found in \cite{Dupuis2007}. Additional numerical results for the algorithms by Blanchet and Li \cite{BL10}, Dupuis et al.\ \cite{Dupuis2007}, and  Asmussen and Kroese \cite{Asmussen2006} can be found in \cite{BL10}. From the existing results it appears as if the algorithm by Dupuis et al.\ \cite{Dupuis2007} has the best performance. Therefore, we only include numerical experiments of  the MCMC estimator and the estimator in \cite{Dupuis2007}, which is labelled (IS).

By construction each simulation run of the MCMC algorithm only generates a single random variable (one simulation step) while both importance sampling and standard Monte Carlo generate $n$ number of random variables ($n$ simulation steps) for the case of fixed number of steps ($N+1$ in the random number of steps case). Therefore the number runs for the MCMC is scaled up to get a fair comparison of the computer runtime between the three approaches.

First consider estimating $\Prob(S_n > a_n)$ where $S_n=Y_1+\cdots+Y_n$ with $Y_{1}$ having a Pareto distribution with density $f_Y(x) = \beta(x+1)^{-\beta-1}$ for $x \geq 0$. Let $a_n = a n$. Each estimate is calculated using $b$ number of batches, each consisting of $T$ simulations in the case of importance sampling and standard Monte Carlo and $T n$ in the case of MCMC. The batch sample mean and sample standard deviation is recorded as well as the average runtime per batch. The results are presented in Table \ref{tab:fixed1}. The convergence of the algorithms can also be visualized by considering the point estimate as a function of number of simulation steps. This is presented in Figure \ref{fig:conv_fix}. The MCMC algorithm appears to outperform the importance sampling algorithm consistently for different choices of the parameters. The improvement over importance sampling appears to increase as the event becomes more rare. This is due to the fact that the asymptotic approximation becomes better and better as the event becomes more rare.

Secondly consider estimating $\Prob(S_N > a_\rho)$ where $S_N=Y_1+\cdots+Y_N$ with $N$ Geometrically distributed $\Prob(N=k)=(1-\rho)^{k-1}\rho$ for $k=1,2,\ldots$ and $a_\rho = a \E[N] = a/\rho$. The estimator considered here is $\widehat{p}_{T} = (\widehat{q}_{T})^{-1}$ with $\widehat{q}_{T}$ as in \eqref{eq:estrandsum}.
Again, each estimate is calculated using $b$ number of batches, each consisting of $T$ simulations in the case of importance sampling and standard Monte Carlo and $T\E[N]$ in the case of MCMC. The results are presented in Table \ref{tab:random1}.
Also in the case of random number of steps the MCMC algorithm appears to outperform the importance sampling algorithm consistently for different choices of the parameters.

We remark that in our simulation with $\rho = 0.2$, $a = 5 \cdot 10^{9}$ the sample standard deviation of the MCMC estimate is zero. This is because we did not observe any indicators $I\{\vee_{j=1}^{n} y_{t,j} > a_{\rho}\}$ being equal to $0$ in this case.

\pagebreak

\bibliographystyle{plain}
\bibliography{references}


\begin{table}[!ht]
\caption{The table displays the batch mean and standard deviation of the estimates of $\Prob(S_n > a_n)$ as well as the average runtime per batch for time comparison. The number of batches run is $b$, each consisting of $T$ simulations for importance sampling (IS) and standard Monte Carlo (MC) and $T  n$ simulations for Markov chain Monte Carlo (MCMC). The asymptotic approximation is $p_{\text{max}} = \Prob(\max\{Y_1,\ldots,Y_n\} > a_n)$.} \label{tab:fixed1}
\begin{tabular}{|c|c|c|c|}
\hline \multicolumn{4}{|c|}{$b=25$, $T=10^5$, $\beta=2$, $n=5$, $a=5$, $p_\text{max}=\text{0.737e-2}$} \\
\hline {} & {MCMC} & {IS} & {MC} \\
\hline {Avg. est.} & {1.050e-2} & {1.048e-2} & {1.053e-2} \\
{Std. dev.} & {3e-5} & {9e-5} & {27e-5} \\
{Avg. time per batch(s)} & {12.8} & {12.7} & {1.4} \\
\hline \multicolumn{4}{|c|}{$b=25$, $T=10^5$, $\beta=2$, $n=5$, $a=20$, $p_\text{max}=\text{4.901e-4}$} \\
\hline {} & {MCMC} & {IS} & {MC} \\
\hline {Avg. est.} & {5.340e-4} & {5.343e-4} & {5.380e-4} \\
{Std. dev.} & {6e-7} & {13e-7} & {770e-7} \\
{Avg. time per batch(s)} & {14.4} & {13.9} & {1.5} \\
\hline \multicolumn{4}{|c|}{$b=20$, $T=10^5$, $\beta=2$, $n=5$, $a=10^3$, $p_\text{max}=\text{1.9992e-7}$} \\
\hline {} & {MCMC} & {IS} & {} \\
\hline {Avg. est.} & {2.0024e-7} & {2.0027e-7} & {} \\
{Std. dev.} & {3e-11} & {20e-11} & {} \\
{Avg. time per batch(s)} & {15.9} & {15.9} & {} \\
\hline \multicolumn{4}{|c|}{$b=20$, $T=10^5$, $\beta=2$, $n=5$, $a=10^4$, $p_\text{max}=\text{1.99992e-9}$} \\
\hline {} & {MCMC} & {IS} & {} \\
\hline {Avg. est.} & {2.00025e-9} & {2.00091e-9} & {} \\
{Std. dev.} & {7e-14} & {215e-14} & {} \\
{Avg. time per batch(s)} & {15.9} & {15.9} & {} \\
\hline \multicolumn{4}{|c|}{$b=25$, $T=10^5$, $\beta=2$, $n=20$, $a=20$, $p_\text{max}=\text{1.2437e-4}$} \\
\hline {} & {MCMC} & {IS} & {MC} \\
\hline {Avg. est.} & {1.375e-4} & {1.374e-4} & {1.444e-4} \\
{Std. dev.} & {2e-7} & {3e-7} & {492e-7} \\
{Avg. time per batch(s)} & {52.8} & {50.0} & {2.0} \\
\hline \multicolumn{4}{|c|}{$b=25$, $T=10^5$, $\beta=2$, $n=20$, $a=200$, $p_\text{max}=\text{1.2494e-6}$} \\
\hline {} & {MCMC} & {IS} & {MC} \\
\hline {Avg. est.} & {1.2614e-6} & {1.2615e-6} & {1.2000e-6} \\
{Std. dev.} & {4e-10} & {12e-10} & {33,166e-10} \\
{Avg. time per batch(s)} & {49.4} & {48.4} & {1.9} \\
\hline \multicolumn{4}{|c|}{$b=20$, $T=10^5$, $\beta=2$, $n=20$, $a=10^3$, $p_\text{max}=\text{4.9995e-8}$} \\
\hline {} & {MCMC} & {IS} & {} \\
\hline {Avg. est.} & {5.0091e-8} & {5.0079e-8} & {} \\
{Std. dev.} & {7e-12} & {66e-12} & {} \\
{Avg. time per batch(s)} & {53.0} & {50.6} & {} \\
\hline \multicolumn{4}{|c|}{$b=20$, $T=10^5$, $\beta=2$, $n=20$, $a=10^4$, $p_\text{max}=\text{5.0000e-10}$} \\
\hline {} & {MCMC} & {IS} & {} \\
\hline {Avg. est.} & {5.0010e-10} & {5.0006e-10} & {} \\
{Std. dev.} & {2e-14} & {71e-14} & {} \\
{Avg. time per batch(s)} & {48.0} & {47.1} & {} \\
\hline
\end{tabular}
\end{table}

\begin{table}[!ht]
\caption{The table displays the batch mean and standard deviation of the estimates of $\Prob(S_N > a_\rho)$ as well as the average runtime per batch for time comparison. The number of batches run is $b$, each consisting of $T$ simulations for importance sampling (IS) and standard Monte Carlo (MC) and $T \,\E[N]$ simulations for Markov chain Monte Carlo (MCMC). The asymptotic approximation is $p_{\text{max}} = \Prob(\max\{Y_1,\ldots,Y_N\} > a_\rho)$.} \label{tab:random1}
\begin{tabular}{|c|c|c|c|}
\hline \multicolumn{4}{|c|}{$b=25$, $T=10^5$, $\beta=1$, $\rho=0.2$, $a=10^2$, $p_\text{max}=\text{0.990e-2}$} \\
\hline {} & {MCMC} & {IS} & {MC} \\
\hline {Avg. est.} & {1.149e-2} & {1.087e-2} & {1.089e-2} \\
{Std. dev.} & {4e-5} & {6e-5} & {35e-5} \\
{Avg. time per batch(s)} & {25.0} & {11.0} & {1.2} \\
\hline \multicolumn{4}{|c|}{$b=25$, $T=10^5$, $\beta=1$, $\rho=0.2$, $a=10^3$, $p_\text{max}=\text{0.999e-3}$} \\
\hline {} & {MCMC} & {IS} & {MC} \\
\hline {Avg. est.} & {1.019e-3} & {1.012e-3} & {1.037e-3} \\
{Std. dev.} & {1e-6} & {3e-6} & {76e-6} \\
{Avg. time per batch(s)} & {25.8} & {11.1} & {1.2} \\
\hline \multicolumn{4}{|c|}{$b=20$, $T=10^6$, $\beta=1$, $\rho=0.2$, $a=5 \cdot 10^7$, $p_\text{max}=\text{2.000000e-8}$} \\
\hline {} & {MCMC} & {IS} & {} \\
\hline {Avg. est.} & {2.000003e-8} & {1.999325e-8} & {} \\
{Std. dev.} & {6e-14} & {1114e-14} & {} \\
{Avg. time per batch(s)} & {385.3} & {139.9} & {} \\
\hline \multicolumn{4}{|c|}{$b=20$, $T=10^6$, $\beta=1$, $\rho=0.2$, $a=5 \cdot 10^9$, $p_\text{max}=\text{2.0000e-10}$} \\
\hline {} & {MCMC} & {IS} & {} \\
\hline {Avg. est.} & {2.0000e-10} & {1.9998e-10} & {} \\
{Std. dev.} & {0} & {13e-14} & {} \\
{Avg. time per batch(s)} & {358.7} & {130.9} & {} \\
\hline \multicolumn{4}{|c|}{$b=25$, $T=10^5$, $\beta=1$, $\rho=0.05$, $a=10^3$, $p_\text{max}=\text{0.999e-3}$} \\
\hline {} & {MCMC} & {IS} & {MC} \\
\hline {Avg. est.} & {1.027e-3} & {1.017e-3} & {1.045e-3} \\
{Std. dev.} & {1e-6} & {4e-6} & {105e-6} \\
{Avg. time per batch(s)} & {61.5} & {44.8} & {1.3} \\
\hline \multicolumn{4}{|c|}{$b=25$, $T=10^5$, $\beta=1$, $\rho=0.05$, $a=5\cdot 10^5$, $p_\text{max}=\text{1.9999e-6}$} \\
\hline {} & {MCMC} & {IS} & {MC} \\
\hline {Avg. est.} & {2.0002e-6} & {2.0005e-6} & {3.2000e-6} \\
{Std. dev.} & {1e-10} & {53e-10} & {55,678e-10} \\
{Avg. time per batch(s)} & {60.7} & {45.0} & {1.3} \\
\hline
\end{tabular}
\end{table}

\pagebreak
\begin{figure}[!ht]
    \centering
        \includegraphics[width=1.0\textwidth]{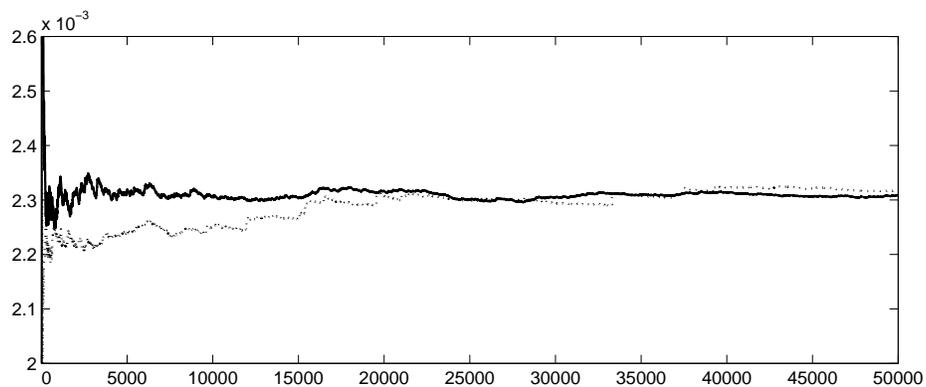}
            \caption{The figure illustrates the point estimate of $\Prob(S_n > a_n)$ as a function of the number of simulation steps, with $n=5$, $a=10$, $\beta=2$. The estimate generated via the MCMC approach is drawn by a \emph{solid line} and the estimate generated via IS is drawn by a \emph{dotted line}.} \label{fig:conv_fix}
\end{figure}


\end{document}